\documentclass[11pt]{article}
\usepackage{tocbibind} 
\usepackage{amsmath,amsthm,amssymb}
\usepackage{fancyhdr}
\usepackage[british]{babel}
\usepackage{geometry,mathtools}
\usepackage{enumitem}
\usepackage{algpseudocode}
\usepackage{dsfont}
\usepackage{centernot}
\usepackage{xstring}
\usepackage{colortbl}

\usepackage{cancel}

\usepackage[section]{algorithm}
\usepackage{graphicx}
\usepackage[font={footnotesize}]{caption}
\usepackage[usenames,dvipsnames,table]{xcolor}
\usepackage{pstricks,tikz}
\usepackage{tkz-euclide}
\tkzSetUpPoint[fill=black, size = 3pt]
\tkzSetUpLine[color=black, line width=0.6pt]
\usepackage{subcaption}
\usepackage[h]{esvect}
\usepackage[
bookmarksopen=true,
bookmarksopenlevel=1,
colorlinks=true,
linkcolor=darkblue,
linktoc=page,
citecolor=darkblue,
urlcolor=darkblue
]{hyperref}
\usepackage{bbm}

\numberwithin{equation}{section}

\definecolor{darkblue}{rgb}{0,0,0.5}

\newdimen\margin
\def\textno#1&#2\par{
   \margin=\hsize
   \advance\margin by -4\parindent
          \setbox1=\hbox{\sl#1}
   \ifdim\wd1 < \margin
      $$\box1\eqno#2$$
   \else
      \bigbreak
      \hbox to \hsize{\indent$\vcenter{\advance\hsize by -3\parindent
      \it\noindent#1}\hfil#2$}
      \bigbreak
   \fi}

\newtheorem{theorem}[algorithm]{Theorem}
\newtheorem{prop}[algorithm]{Proposition}
\newtheorem{lemma}[algorithm]{Lemma}

\newtheorem{fact}[algorithm]{Fact}

\theoremstyle{definition}

\newtheorem{example}[algorithm]{Example}

\newenvironment{claimproof}[1][Proof]
  {\begin{trivlist}
   \item[\hskip\labelsep \textit{#1.}]
   \ignorespaces}
  {\hfill$\blacksquare$\end{trivlist}}

\newcounter{stepenv}
\newenvironment{stepenv}[1][]{\refstepcounter{stepenv}}{}

\newcounter{step}[stepenv]

\newcounter{substep}[step]
\renewcommand{\thesubstep}{\thestep.\arabic{substep}}

\newcounter{claim}[stepenv]
\newenvironment{claim}[1][]{\refstepcounter{claim}\par\medskip\noindent%
        \textit{Claim~\theclaim. #1} \itshape\rmfamily}{\medskip}

\newcommand{\cF}{\mathcal{F}}

\newcommand{\cH}{\mathcal{H}}
\newcommand{\cI}{\mathcal{I}}
\newcommand{\cJ}{\mathcal{J}}
\newcommand{\cK}{\mathcal{K}}

\newcommand{\cP}{\mathcal{P}}

\newcommand{\cS}{\mathcal{S}}

\newcommand{\cY}{\mathcal{Y}}

\newcommand{\bN}{\mathbb{N}}

\newcommand{\EE}{\mathbb{E}}
\newcommand{\PP}{\mathbb{P}}

\def\eps{{\varepsilon}}

\newcommand{\defn}{\emph}

\def\COMMENT#1{}
\def\TASK#1{}
\let\TASK=\footnote             
\let\COMMENT=\footnote          

\begin{document}

\title{Steiner triple systems with high discrepancy}

\author{
Lior Gishboliner \thanks{Department of Mathematics, University of Toronto, Canada.
\emph{Email}:  \href{mailto:lior.gishboliner@utoronto.ca}{\tt lior.gishboliner@utoronto.ca}. In the beginning of this work, LG was supported by SNSF grant 200021\_196965.} 
\and
Stefan Glock \thanks{Fakultät für Informatik und Mathematik, Universität Passau, Germany.
\emph{Email}: \href{mailto:stefan.glock@uni-passau.de}{\tt stefan.glock@uni-passau.de}, \href{mailto:amedeo.sgueglia@uni-passau.de}{\tt amedeo.sgueglia@uni-passau.de}.
SG is funded by the Deutsche Forschungsgemeinschaft (DFG, German Research Foundation) – 542321564,
AS is funded by the Alexander von Humboldt Foundation.}
\and
Amedeo Sgueglia \footnotemark[2]
}

\date{}

\maketitle

\begin{abstract} 
In this paper, we initiate the study of discrepancy questions for combinatorial designs. Specifically, we show that, for every fixed $r\ge 3$ and $n\equiv 1,3 \pmod{6}$, any $r$-colouring of the triples on $[n]$ admits a Steiner triple system of order $n$ with discrepancy $\Omega(n^2)$. This is not true for $r=2$, but we are able to asymptotically characterise all $2$-colourings which do not contain a Steiner triple system with high discrepancy. The key step in our proofs is a characterization of 3-uniform hypergraphs avoiding a certain natural type of induced subgraphs, contributing to the structural theory of hypergraphs.
\end{abstract}

\section{Introduction}

Discrepancy theory aims to answer the following question:
Given a ground set $\Omega$ and a family $\cP \subseteq 2^\Omega$, does there exist an $r$-colouring of $\Omega$ such that each set in $\cP$ contains roughly the same number of elements of each colour?
Formally, let $r \ge 2$ be an integer and $f : \Omega \rightarrow [r]$ be an $r$-colouring of $\Omega$. 
For a set $P \in \cP$ and a colour $c$, let $c(P) := \{x \in P : f(x) = c\}$. 
The discrepancy of $P$ with respect to $f$ is defined as $D_f(P) := r \cdot \nolinebreak \max_{c \in [r]} \left(
|c(P)| - \frac{|P|}{r} \right)$; the larger $D_f(P)$ is, the less balanced is the colouring of $P$. 
The discrepancy of the family $\mathcal{P}$ with respect to $f$ is then defined as $\max_P D_f(P)$. 

Here we consider discrepancy for hypergraphs. In this setting, $\Omega$ is the edge set of a given hypergraph $G$ and $\cP$ is a family of subgraphs of $G$.
Two early results of this type are the works of 
Erd\H{o}s and Spencer~\cite{ES:72} on the discrepancy of cliques in the complete graph, and of Erd\H{o}s, F\"uredi, Loebl and S\'os~\cite{EFLS:95} on the discrepancy of copies of a given spanning tree in the complete graph.
In the last few years, there has been a renewed interest in questions of this type. For ($2$-uniform) graphs, there are by now discrepancy results for many different types of subgraphs, such as perfect matchings and Hamilton cycles~\cite{BCJP:20, FHLT:21,GKM_Hamilton, GKM_trees}, $1$-factorisations~\cite{JFSH:25+}, spanning trees~\cite{GKM_trees}, $H$-factors~\cite{BCPT:21,BCG:23} and powers of Hamilton cycles~\cite{Bradac:22}.
More recently, the analogous question has been studied in the context of hypergraphs for perfect matchings and Hamilton cycles~\cite{BTZ-G:24,GGPS:25+,GGS:23,HLMPSTZ24+,LMX24+}.
Most of these results apply not only in the case of $G$ being the complete graph, but also when $G$ has large minimum degree.

In this paper, we initiate this type of research for combinatorial designs, which form another natural class of spanning hypergraphs. One of the most intensively studied objects in design theory are Steiner triple systems, which also have connections and applications to other areas such as coding theory, geometry, graph decompositions and group theory (see e.g.~\cite{CD:07}).

A \defn{Steiner triple system} (STS for short) of order $n$ is a set $\cS$ of $3$-subsets of $[n]$ such that every $2$-subset of $[n]$ is contained in exactly one of the triples of~$\cS$. Hence, we view an STS as a substructure of the complete $3$-uniform hypergraph $K_{n}^{(3)}$. Due to the required symmetry condition with respect to the $2$-subsets, they are much more complex than the previously studied substructures like perfect matchings and Hamilton cycles. In fact, even the existence question is highly non-trivial.
By a famous result of Kirkman~\cite{kirkman:47}, we know that a Steiner triple system of order $n$ exists if and only if $n\equiv 1,3 \pmod{6}$.
Given the multitude of results concerning discrepancy problems for spanning graphs and hypergraphs, the following question suggests itself:
Assume that $n\equiv 1,3 \pmod{6}$, and consider an $r$-colouring of the $3$-subsets of $[n]$ (corresponding to the edges of $K_n^{(3)}$).
Can we find a Steiner triple system of order $n$ in which one of the colours appears significantly more often than the average?

Recall that an STS of order $n$ has $\binom{n}{2}/3$ triples; therefore, by significant, we mean that the discrepancy is $\Omega(n^2)$. In other words, we would like to have a colour appearing on at least $\left( \frac{1}{r} + \Omega(1) \right)\frac{\binom{n}{2}}{3}$ edges of the STS.
Perhaps surprisingly, the following construction gives a negative answer to this question when $r=2$.

\begin{example}\label{construction}
Partition $[n]$ into two sets $X$ and $Y$, and assign colour blue to all triples touching both $X$ and $Y$, and colour red to all the other triples. In an STS, every edge between $X$ and $Y$ must be covered by exactly one blue triple, and every blue triple covers exactly two of these edges. Hence, in any STS, the number of blue triples is 
$\frac{|X| |Y|}{2}$. 
Let us now choose the sizes of $X,Y$ so that $\frac{|X||Y|}{2}$ is roughly half of the number of edges in an STS. Take $|X| = \lfloor \frac{3+\sqrt{3}}{6}n \rfloor$, so $|X| = \frac{3+\sqrt{3}}{6}n - O(1)$ and $|Y| = n-|X| = \frac{3-\sqrt{3}}{6}n + O(1)$. 
Then 
$\frac{|X||Y|}{2} = \frac{1}{2} \cdot \frac{1}{3}\binom{n}{2} \pm O(n)$, as required.
Thus, in this colouring, any STS has discrepancy $O(n)$, i.e., is roughly balanced.
\end{example}

However, we can show that this is essentially the only $2$-edge-colouring of $K_n^{(3)}$ which does not contain a Steiner triple system with high discrepancy. 

\begin{theorem}
	\label{thm:STS_2_colours}
	For every $\eta >0$, there exists $\mu >0$ such that the following holds for all sufficiently large $n$ with $n \equiv 1,3 \pmod 6$.
	Every $2$-edge-colouring of $K_n^{(3)}$ either contains a Steiner triple system with at least $\frac{1}{2} \cdot \frac{1}{3} \binom{n}{2} + \mu n^2$ edges of the same colour, or it can be obtained from the colouring in Example~\ref{construction} by switching the colour of at most $\eta n^3$ triples.
\end{theorem}

In contrast with the $2$-colour case, when the number of colours is three or more, we can show that a Steiner triple system with high discrepancy always exists.

\begin{theorem}
\label{thm:STS_more_colours}
For all $r\in \bN$ with $r\ge 3$, there exists $\mu>0$ such that the following holds for all sufficiently large $n$ with $n \equiv 1,3 \pmod 6$. Every $r$-edge-colouring of $K_n^{(3)}$ contains a Steiner triple system with at least $\frac{1}{r} \cdot \frac{1}{3} \binom{n}{2} + \mu n^2$ edges of the same colour.
\end{theorem}
\medskip

\textbf{Organisation.}
Section~\ref{sec:overview_and_tools} provides an overview of the various ideas we use to prove our main theorems and collects a list of known tools.
In Section~\ref{sec:key_lemma} we state and prove our key structure theorem for $2$-edge-colourings of $K_n^{(3)}$.
In Sections~\ref{sec:gadgets_to_discrepancy} and~\ref{sec:more_colours}, we show its relevance for finding an STS with high discrepancy.
Section~\ref{sec:main_proofs} contains the proofs of Theorems~\ref{thm:STS_2_colours} and~\ref{thm:STS_more_colours}, and Section~\ref{sec:remarks} consists of concluding remarks and further problems. 
\bigskip

\textbf{Notation.} We use standard graph theory notation.
We denote the complete $3$-uniform hypergraph with $n$ vertices by $K_n^{(3)}$.

For $r \in \mathbb{N}$, we denote $[r]:=\{1,2,\dots,r\}$.
We use the notation $a = b \pm c$ to mean that $b-c \leq a \leq b+c$.
Moreover, for $a, b, c \in (0,1]$, we write $a \ll b \ll c$ in our statements to mean that there are increasing functions $f, g : (0, 1] \to (0, 1]$ such that whenever $a \le f(b)$ and $b \le g(c)$, then the subsequent result holds.

\section{Proof overview and tools}
\label{sec:overview_and_tools}

\subsection{Overview}

In order to obtain an STS with high discrepancy, we will set aside a certain number of small (edge-coloured) graphs, which we call \emph{gadgets}. These will be used to boost the discrepancy.
Let us formalise this notion by defining the gadgets we are going to use.

A {\em Pasch configuration} is the $6$-vertex $3$-graph with edges $135,146,236,245$ (see the left drawing in Figure~\ref{fig:pasch}).
Let $K^{(3)}_{2,2,2}$ denote the complete 3-partite 3-graph with parts of size 2. 
Observe that $K^{(3)}_{2,2,2}$ edge-decomposes into two Pasch configurations: Indeed, suppose that the vertex set of $K^{(3)}_{2,2,2}$ is $[6]$ and the three parts are $\{1,2\}, \{3,4\}$ and $\{5,6\}$. Then $\{135,146,236,245\}$ and $\{136,145,235,246\}$ are two edge-disjoint Pasch configurations covering all the $8$ edges of $K^{(3)}_{2,2,2}$ (see Figure~\ref{fig:pasch}).
Observe that this decomposition is unique.
The crucial property is that both Pasch configurations have the same shadow\footnote{The {\em shadow} of a $3$-graph $P$ is the set of all pairs of vertices which are contained in an edge of $P$.}, namely the $12$ edges of $K_{2,2,2}$.
This means that if an STS contains one of the Pasch configurations, then we can exchange it for the other Pasch configuration, while still having an STS.
An edge-coloured copy $K$ of $K^{(3)}_{2,2,2}$ is called a {\em gadget} if there is a colour $c$ such that the two Pasch configurations decomposing $K$ have a different number of edges in colour $c$. 
We will use gadgets to manipulate the colour profile of an STS.

\begin{figure}[htp!]
\centering
\begin{subfigure}[m]{0.3\textwidth}
    \centering
    \includegraphics{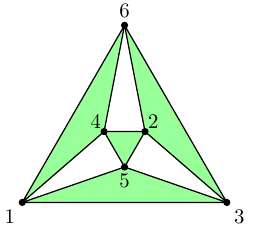}
\end{subfigure}%
\hspace{1cm}
\begin{subfigure}[m]{0.3\textwidth}
    \centering
    \includegraphics{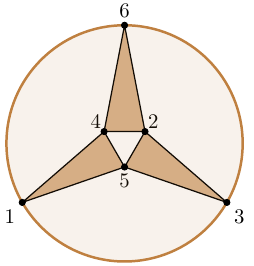}
\end{subfigure}%

\caption{The left figure shows a Pasch configuration. The two figures together show that the edge set of the complete $3$-partite $3$-graph with parts $\{1,2\}$, $\{3,4\}$ and $\{5,6\}$ decomposes into two Pasch configurations: $\{135,146,236,245\}$ (in green) and $\{136,145,235,246\}$ (in brown).}\label{fig:pasch}
\end{figure}

Consider an $r$-edge-colouring of $K_n^{(3)}$ (with $n \equiv 1,3 \pmod 6$). If we can find a collection $\cK$ of $\Omega(n^6)$ many gadgets, then we proceed as follows to show that the colouring contains an STS with high discrepancy (c.f.~Theorem~\ref{thm:gadgets_to_discrepancy}). 
We first pass to a subcollection $\cI \subseteq \cK$ of $\Omega(n^2)$ many gadgets whose shadows are pairwise edge-disjoint and every vertex is contained in at most a small linear number of gadgets from $\cI$ (this will be achieved by sampling every gadget independently with a suitable probability).
Now, consider the $2$-graph obtained from $K_n$ by removing the edges which are in the shadow of a gadget in $\cI$.
This graph still has high minimum degree and thus its edge set can be decomposed into triangles (c.f.~Theorem~\ref{thm:minimum_degree_triangle_decomposition}).
Considering these triangles as edges of $K_n^{(3)}$, we obtain a linear hypergraph\footnote{A hypergraph is \defn{linear} if every two edges intersect in at most one vertex.} covering all edges of $K_n$, except precisely those which are in the shadow a gadget of $\cI$. 
Regardless of the colour of the edges of this linear hypergraph, we can complete it to an STS with high discrepancy by choosing, for each gadget in $\cI$, the appropriate Pasch configuration. 

However, there are edge-colourings of $K_n^{(3)}$ which do not even contain a single gadget: for example, this is the case for the $2$-edge-colouring described in Example~\ref{construction}, where we do not put any condition on the sizes of $X$ and $Y$ (and in particular for a monochromatic colouring, which corresponds to one of the sets $X,Y$ being empty).
As our main structural result (c.f.~Theorem~\ref{thm:gadgets}), we will show that any 2-edge-colouring having few gadgets is close to the colouring given by Example~\ref{construction} (up to switching the two colours).
 
Thus, in the $2$-colour case, we either find enough gadgets (and thus get an STS with high discrepancy, via the argument described above), or we get a partition $X,Y$ as in Example~\ref{construction}.
Now, if the sizes of $X$ and $Y$ differ significantly from those specified in Example~\ref{construction} (i.e., if $|X|/n$ is far from both 
$\frac{3+\sqrt{3}}{6}$ and $1 - \frac{3+\sqrt{3}}{6} = \frac{3-\sqrt{3}}{6}$), then significantly more than half of the edges of $K_n^{(3)}$ have the same colour, in which case a random embedding of any given STS into $K_n^{(3)}$ is expected to have high discrepancy (c.f.~Fact~\ref{fact:random_STS}).
Therefore the only case where we cannot guarantee an STS with high discrepancy is when the colouring is close to the one in Example~\ref{construction} (with the same sizes for $X,Y$). This proves Theorem~\ref{thm:STS_2_colours}.

Finally, let us consider the case of $r \geq 3$ colours. 
Before discussing our strategy, we make the following observation.
Given an $r$-edge-colouring of $K_n^{(3)}$, consider the $2$-edge-colouring of $K_n^{(3)}$ obtained by identifying $r-1$ of the $r$ colours into one colour.
Observe that if a copy of $K^{(3)}_{2,2,2}$ is a gadget for this 2-colouring, then it is also a gadget for the original colouring.
Therefore we can simply refer to a gadget without specifying the colouring.
To take advantage of our structure theorem for 2-colourings, we consider the $2$-edge-colourings of $K_n^{(3)}$ obtained by identifying $(r-1)$-subsets of colours.
If one of these $2$-colourings contains many gadgets, we are done (as, by the observation above, these gadgets are gadgets for the original colouring as well).
Otherwise, by the above discussion, we know the structures of each of them quite precisely.
By comparing them, we will prove that all but at most two colours appear on only very few edges.
Overall, when $r \ge 3$, we always find $\Omega(n^6)$ gadgets unless one colour class is small (c.f.~Theorem~\ref{thm:gadgets_3_colours}). In both cases, this implies that we can find an STS with high discrepancy: in the first case by the argument described at the beginning of this section, and in the second case by embedding any given STS into $K_n^{(3)}$ at random. 

\subsection{Tools}
A \emph{$K_3$-decomposition} of a (2-uniform) graph $G$ is a partition of $E(G)$ into triangles. Thus, an STS of order $n$ is equivalent to a $K_3$-decomposition of $K_n$.
Note that if $G$ has a $K_3$-decomposition then $|E(G)|$ is divisible by $3$ and all vertices of $G$ have even degree. A graph having these properties is called \emph{$K_3$-divisible}.
Kirkman's theorem can then be rephrased as stating that $K_n$ admits a $K_3$-decomposition if and only if $K_n$ is $K_3$-divisible.

As described in the proof overview, our strategy relies on the use of gadgets. 
To make this strategy work, it is crucial to be able to find a $K_3$-decomposition of the graph obtained from $K_n$ by removing the edges in the shadows of the gadgets.
A famous conjecture of Nash-Williams~\cite{nash-williams:70} asks for a minimum degree generalisation of Kirkman's theorem and postulates that every $K_3$-divisible $n$-vertex graph with minimum degree at least $3n/4$ has a $K_3$-decomposition. This is still wide open, but the statement is known to be true if $3/4$ is replaced by a larger constant, which suffices for our purposes.
The following minimum degree version of Kirkman's theorem was proved in~\cite{BKLO:16} by reducing the problem to a fractional variant (see~\cite{BGKLMO:20} for a short proof).

\begin{theorem}[\cite{BKLO:16}]
\label{thm:minimum_degree_triangle_decomposition}
	For sufficiently large $n$, every $K_3$-divisible $n$-vertex graph $G$ with $\delta(G)\ge 0.91n$ has a triangle decomposition.
\end{theorem}

Given an $r$-edge-coloured $K_n^{(3)}$, finding an STS with high discrepancy is straightforward when there is a colour used only few times.
This is formalised in the following fact, for which we provide a proof for completeness.

\begin{fact}
	\label{fact:random_STS}
	For every $\beta >0$ and $r \in \mathbb{N}$ with $r \ge 2$, there exists $\mu >0$ such that the following holds for every $n$ with $n \equiv 1, 3 \pmod 6$.
	Let $K_n^{(3)}$ be $r$-edge-coloured and suppose that there exists a colour $c$ such that the number of edges of colour $c$ is at most $\left(\frac{1}{r} -\beta\right) \cdot \binom{n}{3}$.
	Then there exists a Steiner triple system with at least $\frac{1}{r} \cdot \frac{1}{3} \binom{n}{2} + \mu n^2$ edges of the same colour.
\end{fact}
\begin{proof}
	Let $\mu \ll \beta, 1/r$, fix an arbitrary STS of order $n$ and consider a random embedding into $K_n^{(3)}$ (i.e.~a random bijection from its points to $[n]$).
	Then the expected number of triples of colour $c$ is at most $\frac{1}{3} \binom{n}{2} \cdot \frac{\left(\frac{1}{r} -\beta\right) \cdot \binom{n}{3}}{\binom{n}{3}} =  \left(\frac{1}{r} -\beta\right) \cdot \frac{1}{3} \binom{n}{2}$.
	Therefore there exists an embedding $\cS$ of the STS with at most $\left(\frac{1}{r} -\beta\right) \cdot \frac{1}{3} \binom{n}{2}$ triples of colour $c$.
	By averaging over the other colours, there exists $c' \in [r] \setminus \{c\}$ such that $\cS$ has at least 
	\[
	\frac{\frac{1}{3}\binom{n}{2} - \left(\frac{1}{r} -\beta\right) \cdot \frac{1}{3} \binom{n}{2}}{r-1} \ge \frac{1}{r} \cdot \frac{1}{3} \binom{n}{2} + \mu n^2
	\]
	triples of colour $c'$, as desired.
\end{proof}

We will use the hypergraph removal lemma, first conjectured by Erd\H{o}s, Frankl and R{\"o}dl~\cite{EFR:86}, then proved independently by Gowers~\cite{G:07} and Nagle, R\"{o}dl, Schacht and Skokan~\cite{NRS:06,RS:04,RS:06} and later by Tao~\cite{T:06}.
More precisely, we will use the induced hypergraph removal lemma, proved by R\"{o}dl and Schacht~\cite{RS:09}.

\begin{lemma}[Induced hypergraph removal lemma~\cite{RS:09}]
\label{lemma:removal}
	For every finite family of $k$-graphs $\cF$ and every
	$\alpha > 0$, there exist $\delta > 0$ and $n_0 \in \mathbb{N}$ such that the following holds.
	Let $\cH$ be a $k$-graph on $n \ge n_0$ vertices.
	If for every
	$F \in \cF$, $\cH$ contains at most $\delta n^{v(F)}$ induced copies of $F$, then one can change (add or delete) at most $\alpha n^k$ edges of $\cH$ so that the resulting hypergraph contains no induced copy of any member of $\cF$.
\end{lemma}

\noindent
Finally, we need the following form of the Chernoff bound.

\begin{lemma}[Chernoff bound - Corollaries 2.3 and 2.4 in~\cite{JLR:00}]
	\label{lem:chernoff}
	Let $X$ be the sum of independent Bernoulli random variables. Then for every $\delta \in (0,1)$ we have
	\[ \PP \Big[ |X-\EE[X]| \ge \delta \, \EE[X] \Big] \le 2 \exp \left( -\frac{\delta^2}{3} \EE[X] \right). \]
	Moreover, for every $k \ge 7\, \EE[X]$ we have $\PP[X > k] \le \exp (-k)$. 
\end{lemma}

\section{Key structure theorem}
\label{sec:key_lemma}
In this section we establish our key structural result, characterizing 2-edge-colourings of $K_n^{(3)}$ which have no (or only few) gadgets.
Recall that a gadget is an edge-coloured copy $K$ of $K^{(3)}_{2,2,2}$ for which there is a colour $c$ such that the two Pasch configurations decomposing $K$ have a different number of edges in colour $c$.

\begin{theorem}
	\label{thm:gadgets}
	For any $\alpha >0$, there is $\delta >0$ such that for every $n$ large enough, every $2$-edge-colouring of $K_n^{(3)}$ satisfies (at least) one of the following properties:
	\begin{enumerate}[label=\rm{(\arabic*)}]
		\item There are at least $\delta n^6$ gadgets.
		\item There is a partition $[n] = X \cup Y$ such that after changing the colour of at most $\alpha n^3$ edges, the following holds: All edges inside $X$ and inside $Y$ have the same colour, and all edges intersecting both $X$ and $Y$ have the other colour. 
	\end{enumerate}	
\end{theorem}

	To prove Theorem~\ref{thm:gadgets}, it is enough to prove the following proposition, which makes the stronger assumption that the $2$-edge-colouring of $K_n^{(3)}$ has no gadgets at all. We will then combine this with the induced hypergraph removal lemma (Lemma~\ref{lemma:removal}) to deduce Theorem~\ref{thm:gadgets}.

	\begin{prop}\label{prop:no gadgets}
		In every $2$-edge-colouring of $K_n^{(3)}$ with no gadgets, there is a partition $[n] = X \cup Y$ such that after changing the colour of $O(n^2)$ edges, the following holds: All edges inside $X$ and inside $Y$ have the same colour, and all edges intersecting both $X$ and $Y$ have the other \nolinebreak colour. 
	\end{prop}
	\noindent
	We now give the quick derivation of Theorem~\ref{thm:gadgets} from Proposition~\ref{prop:no gadgets} and Lemma~\ref{lemma:removal}.
    
	\begin{proof}[Proof of Theorem~\ref{thm:gadgets}]
        Let $\mathcal{F}$ be the family of 6-vertex $3$-graphs $F$ such that if the edges of $F$ are coloured red and the non-edges of $F$ are coloured blue, then this 2-colouring of $K_6^{(3)}$ contains a gadget. Then a red/blue colouring of $E(K_n^{(3)})$ has no gadgets if and only if the hypergraph $\mathcal{H}$ of red edges is induced $\mathcal{F}$-free. 
        Choose $\delta>0$ such that Lemma~\ref{lemma:removal} applies with $k=3$ and $\alpha/2$ in place of $\alpha$. 
        
		Now consider any red/blue colouring of $E(K_n^{(3)})$, and let $\mathcal{H}$ be the hypergraph consisting of the red edges. 
		If the colouring has at least $\delta n^6$ gadgets, we are done. Otherwise, $\mathcal{H}$ has at most $\delta n^6$ induced copies of 3-graphs from $\mathcal{F}$. 
		Hence, by Lemma~\ref{lemma:removal}, we can change the colour of at most $\alpha n^3/2$ edges of $K_n^{(3)}$ to obtain a colouring having no induced copies of any $F \in \mathcal{F}$, and hence no gadgets. We can then apply Proposition~\ref{prop:no gadgets} to this colouring.
        Overall, we have changed the colour of at most $\alpha n^3/2 + O(n^2) \leq \alpha n^3$ edges, where the inequality holds since $n$ is large enough.
	\end{proof}

    We now move on to the proof of Proposition~\ref{prop:no gadgets}. We will consider edge-colourings $f \colon E(K_n) \rightarrow \{0,\pm 1\}$. In such a colouring, an {\em unbalanced $C_4$} consists of 4 distinct vertices $a,b,c,d$ and the edges $ab,bc,cd,da$ such that $f(ab) + f(cd) \neq f(bc) + f(ad)$. 
    The following lemma characterizes the structure of $\{0,\pm 1\}$-edge-colourings of $K_n$ with no unbalanced~$C_4$. 
    \begin{lemma}\label{lem:unbalanced C_4}
    	Let $n \geq 50$. Let $f \colon E(K_n) \rightarrow \{0,\pm 1\}$ with no unbalanced $C_4$. Then, after deleting at most $8$ vertices, one of the following holds:
    	\begin{enumerate}[label=\rm{(\arabic*)}]
    		\item All edges have colour $0$.
			\item There is a vertex-partition $A \cup B$ such that all edges inside $A$ have colour $1$, all edges inside $B$ have colour $-1$, and all edges between $A$ and $B$ have colour $0$.
    	\end{enumerate}
    \end{lemma}
	\begin{proof}
		Let $M_1$ (resp.~$M_{-1}$) be a largest matching of edges of colour $1$ (resp.~$-1$). We first claim that every edge inside $V(M_1)$ is coloured with colour $1$, and similarly for $M_{-1}$. So let $ab,cd \in M_1$. If one of the edges between $\{a,b\}$ and $\{c,d\}$ has colour different from $1$, then we get an unbalanced $C_4$. Indeed, without loss of generality, suppose that $f(bc) \neq 1$. Then $f(ab) + f(cd) = 2$ while $f(bc) + f(ad) \leq f(ad) \leq 1$, so $a,b,c,d$ is an unbalanced $C_4$. This shows that $V_1 := V(M_1)$ is a clique in colour $1$. Similarly, $V_{-1} := V(M_{-1})$ is a clique in colour $-1$. 
		Also, by the maximality of $M_1$ and $M_{-1}$, all edges inside $V_0 := [n] \setminus (V_1 \cup V_{-1})$ have colour 0. 
		
		Next, we claim that for all distinct $\alpha,\beta \in \{0,\pm 1\}$, all but at most $3n$ of the edges in $E(V_{\alpha},V_{\beta})$ have colour $\frac{\alpha+\beta}{2}$. 
        (Note that it could happen that $\frac{\alpha+\beta}{2} \not\in \{0,\pm 1\}$.)
        Indeed, suppose otherwise, i.e., that at least $3n+1$ of the edges in $E(V_{\alpha},V_{\beta})$ do not have this colour. 
		Then, by the pigeonhole principle (over the colours), there are at least $n+1$ edges of the same colour, say $\gamma$ with $\gamma \neq \frac{\alpha+\beta}{2}$.
        Among these edges, there must exist two vertex-disjoint ones, say $ab,cd \in E(V_{\alpha},V_{\beta})$. Without loss of generality, $a,d \in V_{\alpha}$ and $b,c \in V_{\beta}$. Moreover, $f(ab)=\gamma=f(cd)$.
        But now 
		$f(ad) + f(bc) = \alpha+ \beta \neq 2\gamma = f(ab) + f(cd)$, meaning that $a,b,c,d$ is an unbalanced $C_4$. This proves our claim.
		
		If $\alpha = 0$ and $\beta \in \{\pm 1\}$, then $\frac{\alpha+\beta}{2}$ is not an integer, so no edge can have colour $\frac{\alpha+\beta}{2}$. This means that 
		$|V_0| \cdot |V_{1}| \leq 3n$ and $|V_0| \cdot |V_{-1}| \leq 3n$. 
		So $|V_0|(n-|V_0|) \leq 6n$. Hence, $|V_0| \leq 6$ or $n-|V_0| \leq 6$ (using $n \geq 50$). In the latter case, after deleting the at most 6 vertices in $[n] \setminus V_0$, all remaining edges have colour $0$, meaning that Item 1 in the lemma holds. 
		
		Suppose now that $|V_0| \leq 6$. Delete the vertices of $V_0$. Set $A := V_1$ and $B := V_{-1} \setminus V_1$, so that $A,B$ are disjoint. By the same argument as above, there is no matching of size $3$ in $E(A,B)$ whose edges have colours $1$ or $-1$ only. Indeed, otherwise, by the pigeonhole principle, there would be two edges between $A$ and $B$ with the same colour $\gamma \neq 0$, and this would give an unbalanced $C_4$. 
		By the K\H{o}nig--Egerv\'ary theorem, we can delete at most 2 vertices such that in the remaining graph, all edges between $A$ and $B$ have colour $0$. Thus, Item 2 in the \nolinebreak lemma \nolinebreak holds.  
	\end{proof}

We also need the following result concerning 2-edge-colourings of $K_n^{(3)}$.
Given (not necessarily disjoint) $A,B,C \subseteq [n]$, we say that an edge $abc$ is of \defn{type} $ABC$ if $a \in A, b \in B, c \in C$.

\begin{lemma}\label{lem:multicolor tight pairs new}
	Let $\chi$ be a $2$-edge-colouring of $K_n^{(3)}$ and let $A, B, C \subseteq [n]$ with $|A|,|B|,|C| \ge 360$. 
    Suppose that every two of the sets $A,B,C$ are disjoint or equal.
    Also, suppose that there are at least $m$ edges of type $ABC$ in each colour. Then there exists $Z \in \{A,B,C\}$ such that there are at least $\frac{m \cdot |Z|}{900}$ pairs of edges $e,e'$ of type $ABC$ with $\chi(e) \neq \chi(e')$, $|e \cap e'| = 2$ and $e \bigtriangleup e' \subseteq Z$.
\end{lemma}

\begin{proof}
	Let the two colours be red and blue, and assume that red is the smallest colour class with respect to edges of type $ABC$.
    Every time we consider an edge, it is implicitly assumed that it is of type $ABC$, and we denote by $e(A,B,C)$ the number of edges of this type.
    We first consider pairs of edges $f = abc, \, f' = a'b'c'$ with $a,a' \in A$, $b,b' \in B$, $c,c' \in C$, $\chi(f) \neq \chi(f')$ and $f \cap f' = \emptyset$. 
    Then the number of such pairs is at least 
	$m \cdot \Big( \frac{e(A,B,C)}{2} - |A||B| - |B||C|-|A||C|\Big) \geq m \cdot \frac{|A||B||C|}{30}$. Indeed, we have at least $m$ choices to choose a red edge, and for each such choice, we then have to choose a blue edge which is vertex-disjoint from the chosen red one. 
    There are at least $\frac{e(A,B,C)}{2}$ blue edges, and at most $|A||B| + |B||C| + |A||C|$ of them intersect the chosen red edge.
    Moreover, $\frac{e(A,B,C)}{2} \ge \frac{|A|(|B|-1)(|C|-2)}{12} \ge \frac{|A||B||C|}{24}$ and
    \begin{align*}
        \frac{e(A,B,C)}{2} - |A||B| - |B||C|-|A||C| &\ge \frac{|A||B||C|}{24} - |A||B| - |B||C|-|A||C| \\
        & \ge |A||B||C| \left(\frac{1}{24}-\frac{1}{|C|}-\frac{1}{|A|}-\frac{1}{|B|}\right) \ge \frac{|A||B||C|}{30}\, ,
    \end{align*}
    where we used that $|A|,|B|,|C| \geq 360$.

    Now, fix a pair $f,f'$ as above.
    Consider the 4 edges 
	$f_1 = f = abc, f_2 = abc', f_3 = ab'c', f_4 = f' = a'b'c'$. Then 
	$|f_i \cap f_{i+1}| = 2$ for every $1 \leq i \leq 3$. As $\chi(f_1) \neq \chi(f_4)$, there must exist $1 \leq i \leq 3$ with $\chi(f_i) \neq \chi(f_{i+1})$. 
    Note that if $i=1$ then $f_i \bigtriangleup f_{i+1}=\{c,c'\}$, if $i=2$ then $f_i \bigtriangleup f_{i+1}=\{b,b'\}$, and if $i=3$ then $f_i \bigtriangleup f_{i+1}=\{a,a'\}$.
    Therefore, in any case, $f_i \bigtriangleup f_{i+1} \subseteq Z$ for some $Z \in \{A,B,C\}$.
    This shows that there are at least 
	$\frac{2}{\binom{6}{3}}\frac{m|A||B||C|}{30} = \frac{m|A||B||C|}{300}$ vertex sets of size six which contain a pair of edges $e,e'$ with $\chi(e) \neq \chi(e')$, $|e \cap e'| = 2$ and $e \bigtriangleup e' \subseteq Z$ for some $Z \in \{A,B,C\}$. 
	The factor of $2 \cdot \binom{6}{3}^{-1}$ accounts for the fact that each such 6-set is counted at most $\frac{1}{2} \cdot \binom{6}{3}$ times. 
    
    By averaging, there exists $Z \in \{A,B,C\}$ such that for at least $\frac{m|A||B||C|}{900}$ of the 6-sets considered above it holds that $e \bigtriangleup e' \subseteq Z$.
    Suppose that $Z=A$; the other cases (i.e., $Z = B$ or $Z = C$) are analogous.
    Each pair $e,e'$ with $|e \cap e'| = 2$ and $e \bigtriangleup e' \subseteq A$ is contained in at most $|B||C|$ vertex sets of size six with two vertices from each part, where we used the assumption that any two of the sets $A,B,C$ are disjoint or equal.
    Hence, the number of pairs $e,e'$ with $\chi(e) \neq \chi(e')$, $|e \cap e'| = 2$ and $e \bigtriangleup e' \subseteq A$ is at least 
    $$\frac{m|A||B||C|/900}{|B||C|} = \frac{m |A|}{900}\, .$$
\end{proof}

We are now ready to prove Proposition~\ref{prop:no gadgets}.
    \begin{proof}[Proof of Proposition~\ref{prop:no gadgets}]
    	We may and will assume that $n$ is large enough. 
    	Let $\chi : E(K_n^{(3)}) \rightarrow \{0,1\}$ be a 2-edge-colouring of $K_n^{(3)}$ with no gadgets. For distinct vertices $x,y \in [n]$, let $f_{xy}$ be the function $f_{xy} : \binom{[n] \setminus \{x,y\}}{2} \rightarrow \{0,\pm 1\}$ given by $f_{xy}(uv) = \chi(xuv) - \chi(yuv)$. Thus, $f_{xy}$ is a colouring with colours $0,\pm 1$ of the (2-uniform) clique on the $n-2$ vertices $[n] \setminus \{x,y\}$. Also, observe that $f_{xy}=-f_{yx}$.
        We first claim that $f_{xy}$ has no unbalanced $C_4$. Indeed, let $a,b,c,d \in [n] \setminus \{x,y\}$ be distinct vertices. Consider the copy of $K^{(3)}_{2,2,2}$ with parts $\{x,y\},\{a,c\},\{b,d\}$. The two Pasch configurations decomposing this copy are $P_1 = \left\{ xab, xcd, yad, ybc  \right\}$ and
    	$P_2 = \left\{ yab, ycd, xad, xbc  \right\}$. Since there is no gadget, we have $\sum_{e\in P_1} \chi(e) = \sum_{e \in P_2} \chi(e)$. On the other \nolinebreak hand,
    	\begin{align*}
    	0 = \sum_{e\in P_1} \chi(e) - \sum_{e \in P_2} \chi(e) &= 
    	\sum_{uv \in \{ab,cd\}} \left( \chi(xuv) - \chi(yuv) \right) - 
    	\sum_{uv \in \{bc,ad\}} \left( \chi(xuv) - \chi(yuv) \right) \\ &= 
    	\sum_{uv \in \{ab,cd\}} f_{xy}(uv) - 
    	\sum_{uv \in \{bc,ad\}} f_{xy}(uv).
    	\end{align*}
    	So $f_{xy}(ab) + f_{xy}(cd) = f_{xy}(bc) + f_{xy}(ad)$. This shows that $f_{xy}$ has no unbalanced $C_4$, as claimed.
    	
    	By Lemma~\ref{lem:unbalanced C_4} (applied to $f_{xy}$), there is a set $S_{xy} \subseteq [n] \setminus \{x,y\}$ of size at most $8$ such that after deleting $S_{xy}$, the colouring $f_{xy}$ satisfies one of the following:
    	\begin{enumerate}
    		\item All edges have colour 0.
    		\item There is a vertex-partition $A_{xy} \cup B_{xy}$ such that all edges inside $A_{xy}$ have colour $1$, all edges inside $B_{xy}$ have colour $-1$, and all edges between $A_{xy}$ and $B_{xy}$ have colour $0$.
    	\end{enumerate}
    	The pair $xy$ is called {\em even} if Item 1 holds and {\em odd} if Item 2 holds. Note that $xy$ is even (resp.~odd) if and only if $yx$ is even (resp.~odd). Also, if $xy$ is odd, we may assume that $A_{yx} = B_{xy}$ and $B_{yx} = A_{xy}$. 
    	
    	\begin{claim}\label{claim:parity}
    		Let $x,y,z \in [n]$ be distinct vertices. If $xy$ and $yz$ have the same parity (i.e., both are even or both are odd) then $xz$ is even. If $xy$ and $yz$ have different parity then $xz$ is odd. 
    	\end{claim}
    	\begin{claimproof}[Proof of Claim~\ref{claim:parity}]
    		Note that restricted to $\binom{[n] \setminus \{x,y,z\}}{2}$, we have 
    		$f_{xz} = f_{xy} + f_{yz}$. Indeed, for every distinct
    		$u,v \in [n] \setminus \nolinebreak \{x,y,z\}$, it holds that 
    		$$
    		f_{xz}(uv) = \chi(xuv) - \chi(zuv) = 
    		\left( \chi(xuv) - \chi(yuv) \right) + \left( \chi(yuv) - \chi(zuv) \right) = f_{xy}(uv) + f_{yz}(uv).
    		$$ 
    		
    		Delete the (at most 17) vertices in $S_{xy} \cup S_{yz} \cup \{y\}$. 
    		Suppose first that $xy,yz$ are even. Then $f_{xy}(uv) = f_{yz}(uv) = 0$ for any two remaining vertices $u,v$ and $f_{xz}(uv) = f_{xy}(uv) + f_{yz}(uv)=0$. This means that $f_{xz} = f_{xy} + f_{yz}$ cannot satisfy Item 2, so it must satisfy Item 1 and hence $xz$ is even.
    		Suppose now that $xy,yz$ have different parity, say $xy$ is odd and $yz$ is even. Then for every $u,v \in A_{xy}$ or $u,v\in B_{xy}$, we have $f_{xz}(uv) = f_{xy}(uv) + f_{yz}(uv) \in \{\pm 1\}$. This means that $f_{xz}$ cannot satisfy Item 1, so it must satisfy Item 2 and $xz$ is odd. Finally, suppose that $xy,yz$ are both odd. One of the four sets $A_{xy} \cap A_{yz}, A_{xy} \cap \nolinebreak B_{yz}, B_{xy} \cap \nolinebreak A_{yz}, B_{xy} \cap \nolinebreak B_{yz}$ has size at least $\frac{n-17}{4}$. For any two $u,v$ in this set, $f_{xy}(uv),f_{yz}(uv) \in \{\pm 1\}$ and thus $f_{xz}(uv) \equiv 0 \pmod 2$. This means that $xz$ cannot satisfy Item 2 (because in that case the edges with colour 0 form a bipartite graph), so $xz$ is even. 
    	\end{claimproof}
       
        Claim~\ref{claim:parity} implies that the even pairs induce disjoint cliques (since if $xy$ and $yz$ are even, then $xz$ is even) and that the number of disjoint even cliques is at most $2$ (since if $xy$ and $yz$ are odd, then $xz$ is even).
        In other words, there is a partition of $[n]$ into $X \cup Y$ such that every pair of vertices which is fully contained in $X$ or fully contained in $Y$ is even, and every other pair is odd.
        We will show that either $X \cup Y$ or $[n] \cup \emptyset$ is the desired partition.
        
        We start by proving that almost all edges of any given type have the same colour. There are four edge types: $XXX$ (i.e., edges entirely contained in $X$); $XXY$ (i.e., edges with two vertices in $X$ and one in $Y$); and similarly $YYX$ and $YYY$. 

        \begin{claim}
        \label{claim:colour_edge_type}
            There exist colours $c_1,c_2,c_3,c_4 \in \{0,1\}$ such that all but at most $10^8|X|^2$ edges of type $XXX$ have colour $c_1$, all but at most $10^8|X|n$ edges of type $XXY$ have colour $c_2$, all but at most $10^8|Y|n$ edges of type $XYY$ have colour $c_3$, and all but at most $10^8|Y|^2$ edges of type $YYY$ have colour $c_4$.
        \end{claim}

        \begin{claimproof}[Proof of Claim~\ref{claim:colour_edge_type}]
            We start with the edges of type $XXX$ and let $m:=10^8|X|^2$.
            First note that we can assume $|X| \ge 360$, as otherwise the total number of edges inside $X$ is smaller than $m$.
            We upper-bound the number of pairs of edges $e,e' \subseteq X$ with $\chi(e) \neq \chi(e')$ and $|e \cap e'| = 2$, as follows. Write $e:=xuv, e':=x'uv$. We first choose $x,x'$. Since $xx'$ is even, if $u,v \notin S_{xx'}$ then $f_{xx'}(uv) = 0$ and hence $\chi(xuv) = \chi(x'uv)$. Therefore we must have $u \in S_{xx'}$ or $v \in S_{xx'}$, and thus there are at most $8|X|$ choices of $uv$ (because $|S_{xx'}| \leq 8$). Hence, the total number of pairs of edges $e,e' \subseteq X$ with $\chi(e) \neq \chi(e')$ and $|e \cap e'| = 2$ is at most $\binom{|X|}{2} \cdot 8|X| < \frac{m|X|}{900}$.
            By Lemma~\ref{lem:multicolor tight pairs new} (applied on input $A=B=C=X$), there must be $c_1 \in \{0,1\}$ such that all but at most $m$ of the edges of type $XXX$ have colour $c_1$. 

            We now move to edges of type $XXY$ and let $m':=10^8|X|n$.
            As before, we can assume $|X|,|Y| \ge 360$, as otherwise the total number of edges of type $XXY$ is smaller than $m'$.
            We bound the number of pairs of edges $e,e'$ of type $XXY$ with $\chi(e) \neq \chi(e')$, $|e \cap e'| = 2$ and $e \bigtriangleup e' \subseteq Y$, as follows. 
            Write $e:=yuv, e':=y'uv$ with $u,v \in X$ and $y,y' \in Y$.
            We first choose $y,y'$. Since $yy'$ is even, in order to have $\chi(e) \neq \chi(e')$, we must have $u \in S_{yy'}$ or $v \in S_{yy'}$, and thus there are at most $8|X|$ choices of $uv$ (because $|S_{yy'}| \leq 8$). 
            Hence, the number of pairs of edges $e,e'$ of type $XXY$ with $\chi(e) \neq \chi(e')$, $|e \cap e'| = 2$ and $e \bigtriangleup e' \subseteq Y$ is at most $\binom{|Y|}{2} \cdot 8|X| < \frac{m'|Y|}{900}$.
            By the same argument, the number of pairs of edges $e,e'$ of type $XXY$ with $\chi(e) \neq \chi(e')$, $|e \cap e'| = 2$ and $e \bigtriangleup e' \subseteq X$ is at most $\binom{|X|}{2} \cdot 8n < \frac{m'|X|}{900}$.
            Therefore, by Lemma~\ref{lem:multicolor tight pairs new} (applied on input $A=B=X$ and $C=Y$), there must be $c_2 \in \{0,1\}$ such that all but at most $m'$ of the edges of type $XXY$ have colour $c_2$. 

            The two remaining types $YYX$ and $YYY$ are symmetric.
        \end{claimproof}

        If $|X| \le 10^{12}$, then the number of edges touching $X$ is at most $|X|n^2 \le 10^{12} n^2$. We may then recolour these $O(n^2)$ edges, as well as at most $O(n^2)$ other edges (using Claim~\ref{claim:colour_edge_type}), to obtain a colouring for which the assertion of the proposition holds (for the partition $X \cup Y$). Thus, we may assume that $|X| > 10^{12}$ and similarly $|Y| > 10^{12}$.

        An edge of type $XXX$ (resp.~$XXY$, $XYY$, $YYY$) is called {\em good} if it has colour $c_1$ (resp.~$c_2,c_3,c_4$). Otherwise, the edge is {\em bad}.
        Sample a set $S$ of six distinct vertices, three from $X$ and three from $Y$, uniformly at random.
        We claim that with positive probability, all $\binom{6}{3} 
        = 20$ edges inside $S$ are good.
        Indeed let $e$ be an edge of $K_n^{(3)}$ with $i$ vertices in $X$ and $3-i$ vertices in $Y$ (so $0 \le i \le 3$).
        The probability that $e \subseteq S$ is 
        \[
            \frac{\binom{|X|-i}{3-i}}{\binom{|X|}{3}} \cdot \frac{\binom{|Y|-(3-i)}{i}}{\binom{|Y|}{3}} \le 729 |X|^{-i} |Y|^{i-3}\, .
        \]
        Using Claim~\ref{claim:colour_edge_type}, the expected number of bad edges $e \subseteq S$ of type $XXX$ is at most $10^8 |X|^2 \cdot 729 |X|^{-3} < 1/4$ (where the inequality uses that $|X| \geq 10^{12}$). 
        By symmetry, the expected number of bad edges $e \subseteq S$ of type $YYY$ is less than $1/4$. 
        Similarly, the expected number of bad edges $e \subseteq S$ of type $XXY$ is at most 
        $10^8 |X| n \cdot 729 |X|^{-2}|Y|^{-1} < 1/4$ (using that, since $|X|,|Y| \ge 10^{12}$ and $n$ is large, we have $|X|\cdot |Y| \ge 10^{12}(n-10^{12}) \ge 10^{12} n/2$).
        By symmetry, the same holds for edges of type $XYY$. By linearity of expectation, the total number of bad edges $e \subseteq S$ is less than 1, and hence there is an outcome in which all edges in $S$ are good. 
                
        Fix a set $S$ for which this holds and write $S \cap X =:\{x_1,x_2,x_3\}$ and $S \cap Y=:\{y_1,y_2,y_3\}$.
        Let $\cP_1:=\{x_1x_2x_3, x_1y_2y_3, y_1x_2y_3, y_1y_2x_3\}$ and $\cP_2:=\{x_1x_2y_3,x_1y_2x_3,y_1x_2x_3,y_1y_2y_3\}$.
        Observe that $\cP_1$ and $\cP_2$ are two edge-disjoint Pasch configurations, $\cP_1$ has exactly one edge of colour $c_1$ and three edges of colour $c_3$, while $\cP_2$ has exactly one edge of colour $c_4$ and three edges of colour $c_2$.
        Moreover $\cP_1 \cup \cP_2$ is a copy of $K^{(3)}_{2,2,2}$ with parts $\{x_i,y_i\}$ for $i=1,2,3$.
        Observe that in order for this copy to not be a gadget, we must have $c_1=c_4$ and $c_2=c_3$.
        
        If $c_1 \neq c_2$, then the partition $X \cup Y$ satisfies the assertion of the proposition.
        Otherwise, $c_1=c_2=c_3=c_4=:c$, and all but at most $4 \cdot 10^8 n^2$ edges of $K_n^{(3)}$ have colour $c$, so we can take the partition $[n] \cup \emptyset$. 
    \end{proof}

\section{From many gadgets to high discrepancy}
\label{sec:gadgets_to_discrepancy}

In this section, we prove that if a colouring of $E(K_n^{(3)})$ contains many gadgets, then it admits a Steiner triple system with high discrepancy.

\begin{theorem}
	\label{thm:gadgets_to_discrepancy}
	For every $r \in \mathbb{N}$ with $r \ge 2$ and $\delta>0$, there exist $n_0 \in \mathbb{N}$ and $\mu >0$ such that for every $n \ge n_0$ with $n \equiv 1,3 \pmod 6$ the following holds.
	Consider an $r$-edge-colouring of $K_n^{(3)}$ and assume that there exists a collection of $\delta n^6$ gadgets.
	Then there is a Steiner triple system with at least $\frac{1}{r} \cdot \frac{1}{3} \binom{n}{2} + \mu n^2$ edges of the same colour.
\end{theorem}

\begin{proof}
	Let $\eps:=10^{-5}$ and assume $n$ is sufficiently large with $n \equiv 1,3 \pmod 6$. Let $\cK$ denote the collection of at least $\delta n^6$ gadgets.
	Recall that the {\em shadow} of a 3-uniform hypergraph $P$ is the (2-uniform) graph consisting of all pairs which are contained in some edge of $P$. Also, two graphs are {\em edge-disjoint} if they share no edges. 
	We refine $\cK$ to a subcollection containing quadratically many gadgets, which do not use the same vertex too much, and whose shadows are pairwise edge-disjoint.
	We will do so by sampling the gadgets with an appropriate probability. 
	
	We first show that, for $p:=\eps n^{-4}$, there exists $\cJ \subseteq \cK$ such that the following hold:
	\begin{enumerate}[label=\upshape(P\arabic*)]
		\item \label{item_P1} $|\cJ| \ge |\cK|p/2$;
		\item \label{item_P2} denoting by $\cY$ the set of pairs of gadgets of $\cJ$ whose shadows are not edge-disjoint, we have $|\cY| < |\cK|p/100$;
		\item \label{item_P3} every vertex of $K_n^{(3)}$ is contained in at most $7\eps n$ gadgets of $\cJ$.
	\end{enumerate} 
	Let $\cJ$ be obtained from $\cK$ by sampling each gadget with probability $p$, independently. We show that with positive probability all the properties~\ref{item_P1},~\ref{item_P2} and~\ref{item_P3} hold.
	First, note that $\EE[|\cJ|] = |\cK|\cdot p \ge \delta \eps n^2$ and thus, by the Chernoff bound (Lemma~\ref{lem:chernoff}), we have $|\cJ| \ge |\cK|p/2$ with probability at least $0.9$.
	Next, observe that there are at most $n^4$ copies of $K_{2,2,2}^{(3)}$ in $K_n^{(3)}$ whose shadow contains a given ($2$-uniform) edge, and that the shadow of $K_{2,2,2}^{(3)}$ contains $12$ edges.
	Therefore, given a gadget $K$, there are at most $12n^4$ other gadgets whose shadow shares an edge with $K$.
	It follows that $\EE[|\cY|] \le |\cK| \cdot (12 n^4) \cdot p^2 \le 12|\cK|p \eps $ and, by Markov's inequality, $\PP[|\cY| \ge |\cK|p/100] \le 1200\eps \le  0.1$.
	Finally, given $v \in V(K_n^{(3)})$, let $X_v$ be the collection of gadgets of $\cJ$ containing the vertex $v$.
	Then $\EE[|X_v|] \le n^5 \cdot p = \eps n$, using that there are at most $n^5$ copies of $K_{2,2,2}^{(3)}$ in $K_n^{(3)}$ containing~$v$.
	Then, using the second part of Lemma~\ref{lem:chernoff}, we have $\PP[|X_v| > 7 \eps n] \le \exp(-7\eps n)$, and by a union bound over the $n$ vertices, with probability at least $0.9$, we have $|X_v| \le 7 \eps n$ for every $v \in V(K_n^{(3)})$.
	Overall, this proves the existence of a suitable $\cJ \subseteq \cK$.
	
	Let $\cJ \subseteq \cK$ satisfy the properties~\ref{item_P1}--\ref{item_P3} and define $\cI$ by removing from $\cJ$ all gadgets appearing in a pair in $\cY$.
	Then $|\cI| \ge |\cK|p/3 \ge \delta \eps n^2/3$, the shadows of the gadgets in $\cI$ are pairwise edge-disjoint and every vertex of $K_n^{(3)}$ is contained in at most $7 \eps n$ gadgets of $\cI$.
	
	Let $G$ be the $n$-vertex (2-uniform) graph obtained from $K_n$ by removing the edges which are in the shadow of a gadget in $\cI$.
    Recall that $n \equiv 1,3 \pmod 6$ and thus $K_n$ is $K_3$-divisible, i.e.~the number of its edges is divisible by $3$ and all its vertices have even degree.
	Since we remove $12$ ($2$-uniform) edges per gadget, $e(G)$ is still divisible by 3. 
    Moreover, for each gadget and each vertex $v$ contained in the gadget, we remove exactly $4$ edges adjacent to $v$ and, since the shadows of the gadgets in $\cI$ are pairwise edge-disjoint, the removed edges adjacent to $v$ are distinct for distinct gadgets. 
	Therefore $G$ is still $K_3$-divisible and $\delta(G) \ge n - 1 - 4 \cdot 7\varepsilon n \ge 0.91 n$.
	It follows from Theorem~\ref{thm:minimum_degree_triangle_decomposition} that $G$ has a triangle decomposition $\mathcal{T}$.
	Observe that the number of triples in $\mathcal{T}$  is exactly
	\begin{equation}
		\label{eq:triangle_decomposition}
		|\mathcal{T}| = \frac{\binom{n}{2} - 12|\cI|}{3} \, .
	\end{equation} 
	
	We now show that by choosing the appropriate Pasch configuration in each gadget of $\cI$, we get an STS with high discrepancy. Note that each Pasch configuration in a $K_{2,2,2}^{(3)}$ forms a triangle decomposition of the shadow of the $K_{2,2,2}^{(3)}$. Hence, choosing a Pasch configuration for each gadget completes the triangle-decomposition $\mathcal{T}$ of $G$ to an STS.

    Let $[r]$ be the colour-set of $E(K_n^{(3)})$.
    For each gadget $I \in \cI$ and colour $c\in [r]$, let $I_c$ be the largest number of edges in colour $c$ among the two Pasch configurations of~$I$. We claim that  
    \begin{equation}\label{eq:I_c}
    \sum_{c \in [r]} I_c \ge 5.
    \end{equation}
    Indeed, let $P_1,P_2$ be the two Pasch configurations in $I$, and let $a_c$ (resp.~$b_c$) be the number of edges of colour $c$ in $P_1$ (resp.~$P_2$). Then $I_c = \max(a_c,b_c)$, so $I_c \geq \frac{a_c + b_c}{2}$. Also, since $I$ is a gadget, there is a colour $c^*$ with $a_{c^*} \neq b_{c^*}$, and so $I_{c^*} > \frac{a_{c^*} + b_{c^*}}{2}$. It follows that $\sum_{c \in [r]}I_c > \frac{1}{2}\sum_{c \in [r]}(a_c + b_c) = \frac{1}{2}(e(P_1) + e(P_2)) = 4$, implying \eqref{eq:I_c}.
    
	For each colour $c \in [r]$, let $T_c$ be the number of triples of colour $c$ in $\mathcal{T}$.
	By choosing for each $I \in \cI$ the Pasch configuration with $I_c$ edges of colour $c$, we get a Steiner triple system with $S_c := T_c+ \sum_{I \in \cI} I_c$ edges in colour $c$. We have
	\begin{equation*}
		\sum_{c \in [r]} S_c = \sum_{c \in [r]} T_c + \sum_{c \in [r]} \sum_{I \in \cI} I_c
		\ge \frac{\binom{n}{2} - 12|\cI|}{3} + 5 \cdot |\cI| = \frac{1}{3}\binom{n}{2}+|\cI| \geq 
		\frac{1}{3}\binom{n}{2} + \frac{\delta \varepsilon n^2}{3},
	\end{equation*}
	where the first inequality uses~\eqref{eq:triangle_decomposition} and~\eqref{eq:I_c}.
	By averaging, there is a colour $c \in [r]$ such \nolinebreak that 
	\[
	S_c \ge \frac{1}{r}\cdot \left[ \frac{1}{3} \binom{n}{2} + \frac{\delta \eps n^2}{3} \right]\, ,
	\]
	which proves the theorem with $\mu:=(\delta\varepsilon)/(3r)$. 
\end{proof}

\section{Edge-colourings with more than two colours}
\label{sec:more_colours}
	Here we use the structure theorem (Theorem~\ref{thm:gadgets}) to show that for $r \geq 3$, if an $r$-edge-colouring of $K_n^{(3)}$ does not have many gadgets, then all but at most two colour classes are tiny.

	\begin{theorem}
	\label{thm:gadgets_3_colours}
		Let $r \in \mathbb{N}$ with $r \ge 3$ and $\zeta >0$.
		Then there exists $\delta>0$ such that for every $r$-edge-colouring of $K_n^{(3)}$, (at least) one of the following properties holds:
		\begin{enumerate}[label=\rm{(\arabic*)}]
            \item \label{item:few_edges_colour_c} There are two colours $c^*,d^*$ such that at most $\zeta n^3$ edges have colour different from $c^*,d^*$;
			\item \label{item:few_gadgets} There are at least $\delta n^6$ gadgets.
		\end{enumerate}
	\end{theorem}

	\begin{proof}		
		Fix a new constant $\alpha$ such that $\delta \ll \alpha \ll \zeta, 1/r$.
		Let the colour-set be $[r]$ and assume that Item~\ref{item:few_gadgets} in the statement of the theorem does not hold, i.e.~that $K_n^{(3)}$ contains less than $\delta n^6$ gadgets.
		For a colour $c \in [r]$, let $\chi_c : E(K_n^{(3)}) \rightarrow \{c,\bar{c}\}$ be the colouring of $K_n^{(3)}$ which identifies the colours in $[r] \setminus \{c\}$ into one (new) colour $\bar{c}$. I.e., $\chi_c(e) = c$ if $\chi(e) = c$ and $\chi_c(e) = \bar{c}$ otherwise.
		Then $\chi_c$ has less than $\delta n^6$ gadgets, because a gadget in $\chi_c$ is also a gadget in the original colouring.
		Hence, by Theorem~\ref{thm:gadgets}, there is a partition $[n] = X_c \cup Y_c$ such that 
		after changing the colours of at most $\alpha n^3$ edges, all edges inside $X_c$ and $Y_c$ have colour $a_c$, and all edges intersecting both $X_c$ and $Y_c$ have colour $b_c$, where $\{a_c,b_c\} = \{c,\bar{c}\}$. We call $a_c$ the colour {\em inside} the parts $X_c,Y_c$, and $b_c$ the colour {\em across} the partition $(X_c,Y_c)$.
        
        In particular, by considering the Venn diagram of the partitions $(X_c,Y_c)_{c \in [r]}$ and averaging, we get a set $Z$ of size at least $n/2^r$ such that $Z \subseteq X_c$ or $Z \subseteq Y_c$ for each $c \in [r]$.
        By averaging, there is a colour $c^*$ such that at least $\frac{1}{r} \binom{|Z|}{3} \ge 2 \alpha n^3$ edges inside $Z$ have colour $c^*$.
        Therefore $c^*$ is the colour inside the part $X_{c^*}$ and $Y_{c^*}$, that is~$a_{c^*}=c^*$.
        If $|X_{c^*}| \le \zeta n/2$, then the number of edges of colour different from $c^*$ (in $\chi_{c^*}$ and thus in the original colouring of $K_n^{(3)}$) is at most $|X_{c^*}| n^2 + \alpha n^3 \leq \zeta n^3/2 + \alpha n^3 < \zeta n^3$, so Item~\ref{item:few_edges_colour_c} in the statement of the theorem holds (in fact, all colour classes but one are tiny).
        Similarly, Item~\ref{item:few_edges_colour_c} holds if $|Y_{c^*}| \le \zeta n/2$.
        Therefore, from now on, we can assume $|X_{c^*}|,|Y_{c^*}| > \zeta n/2$.

        The number of edges of colour distinct from $c^*$ (in $\chi_{c^*}$ and thus in the original colouring) is at least \begin{align*}
            \binom{|X_{c^*}|}{2} \cdot |Y_{c^*}| + |X_{c^*}| \cdot \binom{|Y_{c^*}|}{2} - \alpha n^3 &= \frac{|X_{c^*}| \cdot |Y_{c^*}| \cdot (n-2)}{2} - \alpha n^3 \\
            & \ge \zeta n^3/16 - \alpha n^3 \ge \zeta n^3/20\, ,
        \end{align*}
        where we used that $|X_{c^*}| + |Y_{c^*}| =n$, $|X_{c^*}|,|Y_{c^*}| > \zeta n/2$ and so $|X_{c^*}| \cdot |Y_{c^*}| > \zeta n^2/4$.
        By averaging, there is a colour $d^* \in [r] \setminus \{c^*\}$ such that the number of edges of colour $d^*$ is at least $\zeta n^3/(20r)$.

        We now show that $d^*$ is the colour across in $\chi_{d^*}$. (In fact, the same argument shows that for any $d \in [r] \setminus \{c^*\}$, $d$ is the colour across in $\chi_d$.)
        Let $A:=X_{c^*} \cap X_{d^*}$, $B:=X_{c^*} \cap Y_{d^*}$, $C:=Y_{c^*} \cap X_{d^*}$ and $D:=Y_{c^*} \cap Y_{d^*}$.
        Then at least one of these four sets has size at least $n/4$. Without loss of generality, suppose that $|A| \geq n/4$. 
        Since $A \subseteq X_{c^*}$ and $a_{c^*}=c^*$, the number of edges of colour $c^*$ inside $A$ is at least $\binom{|A|}{3}-\alpha n^3 \ge 2\alpha n^3$.
        Since $d^* \neq c^*$ and $A \subseteq X_{d^*}$, then $d^*$ cannot be the colour inside the parts in $\chi_{d^*}$, as wanted.        
        
        Moreover, we must have $|X_{d^*}|,|Y_{d^*}| > \zeta n/(5r)$ as otherwise, by a similar calculation as above, the number of edges of colour $d^*$ would be at most $\binom{|X_{d^*}|}{2} \cdot |Y_{d^*}| + |X_{d^*}| \cdot \binom{|Y_{d^*}|}{2} + \alpha n^3 < \zeta n^3/(20r)$, contradicting the choice of $d^*$.
        Next, we show that the partitions $(X_{c^*},Y_{c^*})$ and $(X_{d^*},Y_{d^*})$ do not differ too much, in the following sense.
	\begin{claim}
	\label{claim:partitions_are_related}
		Let $m:=\zeta n/(10r)$.
            We have $|A| < m$ or 
		$|B| < m$. 	
		Similarly, $|C| < m$ or 
		$|D| < m$.
	\end{claim}
	\begin{claimproof}[Proof of Claim~\ref{claim:partitions_are_related}]
            We only prove the first statement, the second one can be proved similarly.
            So suppose by contradiction that $|A|,|B| \geq m$. As $b_{d^*} = d^*$, the number of the edges across $(A,B)$ of colour $d^*$ is at least $\binom{|A|}{2} \cdot |B| - \alpha n^3 \ge 2 \alpha n^3$. On the other hand, as $A \cup B \subseteq X_{c^*}$ and $a_{c^*} = c^*$, all but at most $\alpha n^3$ of the edges inside $A \cup B$ have colour $c^*$, a contradiction. 
		\end{claimproof}        
	
	As already mentioned, we can assume that $|A| \ge n/4$.
        Then, by Claim~\ref{claim:partitions_are_related}, $|B| < \zeta n/(10 r)$.
        Since $B \cup D = Y_{d^*}$ and $|Y_{d^*}| > \zeta n/(5r)$, we must have $|D| > \zeta n /(10r)$.
        Therefore, again by Claim~\ref{claim:partitions_are_related}, we have $|C| < \zeta n /(10r)$. Then $|B \cup C| < \zeta n/(5r)$ and the number of edges touching $B \cup C$ is at most $|B \cup C| \cdot n^2 < \zeta n^3/(5r)$.
        As $a_{c^*} = c^*$, all but $\alpha n^3$ of the edges inside $A$ and inside $D$ have colour $c^*$. Also, as $b_{d^*} = d^*$, all but at most $\alpha n^3$ of the edges across $(A,D)$ have colour $d^*$. This means that the number of edges whose colour is neither $c^*$ nor $d^*$ is at most $2\alpha n^3 + \zeta n^3/(5r) < \zeta n^3$. 
        In particular, Item~\ref{item:few_edges_colour_c} in the statement of the theorem holds.
	\end{proof}

\section{Proofs of main theorems}
\label{sec:main_proofs}
We are now ready to prove our main theorems, concerning Steiner triple systems with high discrepancy in $r$-edge-colourings of $K_n^{(3)}$. 
We start with the case $r \ge 3$.

\begin{proof}[Proof of Theorem~\ref{thm:STS_more_colours}]
	Let $r \ge 3$ and 
	$
	1/n \ll \mu \ll \delta \ll \zeta, 1/r
	$. 
	Let $K_n^{(3)}$ be edge-coloured in $[r]$.
	Then by Theorem~\ref{thm:gadgets_3_colours}, there exists a colour $c$ such that the number of edges of colour $c$ is at most $\zeta n^3$, or there are at least $\delta n^6$ gadgets.
	In the first case, the result follows from Fact~\ref{fact:random_STS}. In the second case, we conclude using Theorem~\ref{thm:gadgets_to_discrepancy}.	
\end{proof}
\noindent 
Finally, we move to the case $r=2$.
\begin{proof}[Proof of Theorem~\ref{thm:STS_2_colours}]
	Given $\eta$, define new constants such that 
	\[
		1/n \ll \mu \ll \delta \ll \alpha \ll \beta \ll \gamma \ll \eta, 1/r\, .
	\]
	Let $K_n^{(3)}$ be $2$-edge-coloured. 
	By Theorem~\ref{thm:gadgets}, there are at least $\delta n^6$ gadgets or there is a partition $[n] = X \cup Y$ such that after changing the colour of at most $\alpha n^3$ edges, which we call the {\em bad} edges, all edges inside $X$ or inside $Y$ have the same colour, say red, and all edges intersecting both $X$ and $Y$ have the \nolinebreak other \nolinebreak colour, \nolinebreak say blue. 
	In the first case (i.e., if there are at least $\delta n^6$ gadgets), Theorem~\ref{thm:gadgets_to_discrepancy} immediately gives an STS with discrepancy at least $\mu n^2$.
	
	In the second case, we consider the sizes of $X$ and $Y$. Put $x := |X|/n$, so $|Y| = (1-x)n$. 
	Put $z := x^3 + (1-x)^3$.
	Then the number of red edges in $K_n^{(3)}$ is $(z + o(1))\binom{n}{3} \pm \alpha n^3 = 
	(z \pm 7\alpha)\binom{n}{3}$. Hence, if $|z - 1/2| \geq 2\beta$, then, as $\beta \gg \alpha$, the number of red edges is at least $(1/2+\beta)\binom{n}{3}$ or at most 
	$(1/2-\beta)\binom{n}{3}$.
	Now Fact~\ref{fact:random_STS} gives an STS with at least $\frac{1}{2} \cdot \frac{1}{3} \binom{n}{2} + \mu n^2$ edges of the same colour. 
	
	Suppose now that $|z - 1/2| < 2\beta$.
	The solutions to $x^3 + (1-x)^3 = 1/2$ are $x = \frac{3 \pm \sqrt{3}}{6} \in [0,1]$. By continuity, and as $\beta \ll \gamma$, we have that 
	$|x - \frac{3 + \sqrt{3}}{6}| \leq \gamma$ or 
	$|x - \frac{3 - \sqrt{3}}{6}| \leq \gamma$. These two cases are symmetric (by replacing $x$ with $1-x$), so assume that
	$|x - \frac{3 + \sqrt{3}}{6}| \leq \gamma$. Then, by moving $O(\gamma n)$ vertices between $X$ and $Y$, we can make sure that $|X| = \lfloor \frac{3+\sqrt{3}}{6}n \rfloor$. By recolouring the edges touching the moved vertices, as well as the at most $\alpha n^3$ bad edges, we obtain that all edges inside $X$ and inside $Y$ are red and all edges intersecting both $X,Y$ are blue. As $\alpha,\gamma \ll \eta$, this recolours at most $\eta n^3$ edges. After this recolouring, we have the construction given by Example~\ref{construction}, as required. 
\end{proof}

\section{Concluding remarks}
\label{sec:remarks}

It would be interesting to extend our results to general designs. 
Given $k,\ell \in \mathbb{N}$, a \emph{$(k,\ell)$-Steiner system} of order $n$ is a set $\cS$ of $k$-subsets of $[n]$ such that every $\ell$-subset of $[n]$ is contained in exactly one of the sets of $\cS$. Thus, an STS corresponds to $(k,\ell) = (3,2)$. When general Steiner systems exist was a long-standing open problem, until finally settled by Keevash~\cite{keevash:14}, who showed that they exist whenever $n$ is sufficiently large and satisfies a necessary ``divisibility condition''. An alternative proof of Keevash's result and a minimum degree version analogous to Theorem~\ref{thm:minimum_degree_triangle_decomposition} was proved in~\cite{GKLO:23}. 
It would be interesting to understand for which values of $k,\ell,r$ it holds that every $r$-edge-colouring of $K_n^{(k)}$ contains a $(k,\ell)$-Steiner system of order $n$ with high discrepancy (assuming that $n$ is such that a $(k,\ell)$-Steiner system of order $n$ exists).

We remark that an approach to investigate general Steiner systems could be via characterization results for specific hypergraph cut properties. 
In fact, this was our first strategy and gave a working proof of our main result, which later inspired our cleaner key structure theorem (Theorem~\ref{thm:gadgets}).
We provide here some details for the benefit of the reader. Fix a $2$-colouring of the edges of $K_n^{(3)}$.
If there is an equipartition of $[n]$ into three sets $A, B, C$ 
such that, among the edges which intersect exactly two of the sets $A, B, C$, some colour appears significantly more than half of the times, then using known nibble-type results and the absorption technique one could find a Steiner triple system with high discrepancy. 
Hence, we can assume that the colouring has the following \defn{cut property}: For every equipartition $A, B, C$ of $[n]$ each colour appears roughly the same number of times on edges which intersect exactly two of the sets $A, B, C$. 
By combining probabilistic and algebraic arguments with results from the theory of association schemes, in a similar way as in the work of Shapira and Yuster~\cite{SY:12}, we could characterise all the colourings having this cut property. 

\section*{Acknowledgements}
We thank the anonymous referees for their valuable comments.

\bibliographystyle{amsplain_v2.0customized}
\bibliography{References}

\end{document}